\documentclass[12pt,oneside]{amsart}
\usepackage{amssymb,verbatim,graphicx,epstopdf}

\numberwithin{equation}{section}
\newtheorem{theorem}{Theorem}[section]
\newtheorem{lemma}[theorem]{Lemma}
\newtheorem{proposition}[theorem]{Proposition}
\newtheorem{corollary}[theorem]{Corollary}

\newtheorem{algorithm}[theorem]{Algorithm}

\theoremstyle{definition}
\newtheorem{definition}[theorem]{Definition}
\newtheorem{remark}[theorem]{Remark}
\newtheorem{example}[theorem]{Example}
\newtheorem{construction}[theorem]{Construction}

\newtheorem{notation}[theorem]{Notation}
\newtheorem{question}[theorem]{Question}

\DeclareMathOperator{\Ann}{Ann}
\DeclareMathOperator{\Ass}{Ass}

\DeclareMathOperator{\borel}{Borel}
\DeclareMathOperator{\Borel}{Borel}
\DeclareMathOperator{\sfBorel}{sfBorel}
\DeclareMathOperator{\sfborel}{sfBorel}

\DeclareMathOperator{\gens}{gens}
\DeclareMathOperator{\Bgens}{Bgens}
\DeclareMathOperator{\trunc}{trunc}
\DeclareMathOperator{\supp}{supp}
\DeclareMathOperator{\codim}{codim}
\DeclareMathOperator{\pd}{pd}

\DeclareMathOperator{\HS}{HS}

\newcommand{\bgeq}{\succ}

\begin{document}


\title{Borel generators}

\author{Christopher A. Francisco}
\address{Department of Mathematics, Oklahoma State University,
401 Mathematical Sciences, Stillwater, OK 74078}
\email{chris@math.okstate.edu}
\urladdr{http://www.math.okstate.edu/$\sim$chris}

\author{Jeffrey Mermin}
\address{Department of Mathematics, Oklahoma State University,
401 Mathematical Sciences, Stillwater, OK 74078}
\email{mermin@math.okstate.edu}
\urladdr{http://www.math.okstate.edu/$\sim$mermin}

\author{Jay Schweig}
\address{Department of Mathematics, University of Kansas, 405 Snow Hall, Lawrence, KS 66045}
\email{jschweig@math.ku.edu}
\urladdr{http://www.math.ku.edu/$\sim$jschweig}

\thanks{Version: \today}

\begin{abstract}
We use the notion of Borel generators to give alternative methods for computing standard invariants, such as associated primes, Hilbert series, and Betti numbers, of Borel ideals. Because there are generally few Borel generators relative to ordinary generators, this enables one to do manual computations much more easily. Moreover, this perspective allows us to find new connections to combinatorics involving Catalan numbers and their generalizations. We conclude with a surprising result relating the Betti numbers of certain principal Borel ideals to the number of pointed pseudo-triangulations of particular planar point sets.
\end{abstract}

\maketitle


\section{Introduction} \label{s:intro}

Borel-fixed ideals are arguably the most important ideals in computational commutative algebra. Their combinatorial properties make them easier to investigate than arbitrary monomial ideals, and thanks to work of Galligo \cite{Ga} and Bayer and Stillman \cite{BS}, we know that generic initial ideals are always Borel-fixed. Moreover, Eliahou and Kervaire show that the minimal graded free resolution of a Borel-fixed ideal in characteristic zero has a particularly nice form \cite{EK}. Thus, as Bayer and Stillman prove in \cite{BS-invent}, one can determine the regularity of any homogeneous ideal simply by computing the reverse-lex generic initial ideal and determining the highest degree of a minimal generator. Furthermore, lexicographic ideals are Borel-fixed, and thus, by understanding Borel-fixed ideals, we gain valuable insight into the Hilbert functions and graded Betti numbers of arbitrary homogeneous ideals. In addition, papers such as \cite{CE,GHP,Si-pborel} exploit the resolutions of Borel-fixed ideals to obtain minimal free resolutions of closely related monomial ideals as well. Borel-fixed ideals are also of special importance in geometric combinatorics, where they arise in connection with shifted simplicial complexes.

Let $S=k[x_1,\dots,x_n]$, where $k$ is a field. We use the term \emph{Borel ideal} to mean a strongly stable ideal, also known as a 0-Borel-fixed ideal, which is an ideal that is fixed by the Borel group in characteristic zero. Thus the collection of Borel ideals is precisely the collection of generic initial ideals over a field of characteristic zero.

The usual method for studying a Borel ideal is to compute invariants of the ideal in terms of its minimal monomial generating set. In contrast, a few researchers have instead concentrated on the \emph{Borel generators} of a Borel ideal $B$, a subset $T$ of monomials in $B$ such that every minimal monomial generator of $B$ can be obtained from Borel moves on elements of $T$. For example, Herzog and Srinivasan \cite{HS} prove the Borel case of the Multiplicity Conjectures using Borel generators.  Many of the applications of this idea involve principal Borel ideals (ideals with a single Borel generator) in special situations. For example, in \cite{Si}, Sinefakopoulos constructs a shellable polytopal cell complex supporting a minimal free resolution of a principal Borel ideal. J\"ollenbeck and Welker also construct a minimal cellular resolution of a principal Borel ideal in \cite{JW}. In addition, Bonanzinga classifies the principal Borel ideals that are lexicographic and investigates which principal Borel ideals are Gotzmann \cite{Bo}. There have also been a number of papers that include results on principal $p$-Borel ideals, including \cite{AH,HPV,JW,Pa,Po}.

Our goal in this paper is a bit different. We use the notion of minimal Borel generators of a Borel ideal to develop alternative ways of computing standard invariants in commutative algebra. Our approach has two primary advantages over the traditional ways of doing the computations. First, the number of Borel generators is usually far smaller than the number of ordinary minimal generators of a Borel ideal. Hence when computing invariants of Borel ideals by hand, it is often much easier to use our methods because one does not need to keep track of all the minimal monomial generators of the ideal. This is especially true when working with a principal Borel ideal. Second, our different point of view allows us to uncover connections to combinatorics and computational geometry that are not easily visible using standard techniques. It becomes natural to study principal Borel ideals, which do not appear any more interesting than other Borel ideals from the usual perspective. Using our methods, however, the Catalan numbers and generalizations arise naturally. Investigating principal Borel ideals leads us to an interesting sequence of Betti numbers with a strong connection to pointed pseudo-triangulations.  Throughout the paper, we discuss how to apply our techniques to squarefree Borel ideals as well.

Our paper is organized as follows. In Section~\ref{s:prelim}, we explain our terminology and describe how to do some fundamental operations on Borel ideals using only Borel generators. We give two methods for computing the associated primes of a Borel ideal in terms of its Borel generators in Section~\ref{s:primes}, and we compute the Alexander dual in the squarefree case. In Section~\ref{s:stanley}, we determine a Stanley decomposition for $S/B$, where $B$ is a Borel ideal, which yields a short proof of Stanley's conjecture in this case.  The decomposition also provides formulas for the Hilbert series and multiplicity of $S/B$. We explain how Catalan numbers and their generalizations arise in the Hilbert functions (and subsequently, the Betti numbers) of Borel ideals in Section~\ref{s:catalan}. In Section~\ref{s:Betti}, we describe how to compute the graded Betti numbers of a Borel ideal using only the Borel generators, and we describe several Poincar\'e series associated to $B$, where $B$ is a Borel or squarefree Borel ideal. Finally, in Section~\ref{s:ppt}, we prove an unexpected connection between the Betti numbers of certain principal Borel ideals and pointed pseudo-triangulations of particular planar point sets studied in \cite{AOSS}. While we often focus on principal Borel ideals, we illustrate how to use our methods for these ideals iteratively to do computations for general Borel ideals.

We gratefully acknowledge the computer algebra system Macaulay 2 \cite{M2}, which we used to compute examples throughout this project, and thank Dana Brunson for use of the High Performance Computing Center at Oklahoma State University. The first author was partially supported by an NSA Young Investigator Grant.

\section{Preliminaries} \label{s:prelim}

Throughout, let $S=k[x_1,\dots,x_n]$, where $k$ is a field.

\begin{definition}
Let $m$ be a monomial in $S$. A \emph{Borel move} on $m$ is an operation that sends $m$ to a monomial $m \cdot \frac{x_{i_1}}{x_{j_1}} \cdots \frac{x_{i_s}}{x_{j_s}}$, where $i_t < j_t$ for all $t$, and all $x_{j_t}$ divide $m$.
\end{definition}

\begin{definition}
A monomial ideal $B$ is a \emph{Borel ideal} if $B$ is closed under Borel moves. That is, if $m \in B$, then any monomial obtained from a Borel move on $m$ is also in $B$.
\end{definition}

When $B$ is a Borel ideal, Bayer and Stillman show in \cite[Corollary 2]{BS} that $\Ass(S/B)$ has an especially nice structure. 

\begin{theorem}\label{t:bayerstillman}
Let $B$ be a Borel ideal in $S$. If $P \in \Ass(S/B)$, then $P=(x_1,\dots,x_i)$ for some $i \le n$.
\end{theorem}

\begin{notation} \label{n:factorization}
Given a monomial $m$ of degree $d$, we may write $m$ uniquely in the form $m=\prod_{j=1}^{d}x_{i_{j}}$ with $i_{1}\leq i_{2}\leq \dots \leq i_{d}$. We call this expression the \emph{factorization} of $m$. We say the variable $x_{i_{r}}$ is in the \emph{$r^{\text{th}}$ position}, and $\max(m)=i_d$. We will routinely abuse notation and write $\max(m)=x_{i_d}$. We define the minimum similarly, setting $\min(m)=i_1$ (or $x_{i_1}$).
\end{notation}

\begin{example}\label{e:position}
The factorization of $a^2cd^3ef$ is $aacdddef$. The variable in the seventh position is $e$, and $\max(a^2cd^3ef)=f$. 
\end{example}


\begin{definition} \label{d:borelideal}
Let $T=\{m_1, \dots, m_s\}$ be a set of monomials. Define $\Borel(T)=\Borel(m_1,\dots,m_s)$ to be the smallest Borel ideal containing $T$.  We say that $m_1,\dots,m_s$ are \emph{Borel generators} of $\Borel(T)$.  If $T=\{m\}$ has cardinality one, we say that $\Borel(T)$ is the \emph{principal Borel ideal} generated by $m$.
\end{definition}

\begin{remark} \label{r:m2}
The computer algebra system Macaulay 2 \cite{M2} has a method for determining the smallest Borel ideal containing a set of monomials $m_1, \dots, m_s$. Typing {\tt borel monomialIdeal(a*b*e,a*c*d)}, for example, produces the smallest Borel ideal containing the monomials $abe$ and $acd$.
\end{remark}

\begin{definition} \label{d:precedes}
Factor $m_{1}=\prod_{j=1}^{r} x_{i_{j}}$ and $m_{2}=\prod_{j=1}^{s}x_{k_{j}}$.  We say that $m_{1}\bgeq m_{2}$ if $r\geq s$ and $i_{j}\leq k_{j}$ for all $j\leq s$.  In this case, we say that \emph{$m_{1}$ precedes $m_{2}$ in the Borel order}.
\end{definition}

\begin{remark} \label{r:lex}
If $m_2 \mid m_1$, then $m_1 \bgeq m_2$.   If $m_1 \bgeq m_2$, then $m_1\geq_{\text{Lex}} m_2$, where $\geq_{\text{Lex}}$ is the (ungraded) lexicographic order.
\end{remark}

\begin{example} \label{e:precedes}
$a^2$ and $ab$ precede $b^2$, as does $a^3$, but $a$ and $b^2$ are incomparable in the Borel order.
\end{example}

\begin{lemma} \label{l:borelvs}
$\Borel(T)$ is spanned as a vector space by the monomials which precede some element of $T$ in the Borel order.
\end{lemma}


\begin{proposition} \label{p:uniqueborelgens}
Every Borel ideal $B$ has a unique minimal set of Borel generators.  We call this set $\Bgens(B)$. 
\end{proposition}

\begin{proof}
Given a Borel ideal $B$, let $\gens(B)=\{m_1,\dots,m_s\}$ be the minimal monomial generators of $B$.  This is a finite set, partially ordered by the Borel order.  Let $T$ be the minimal monomials in this poset.  Clearly $B=\Borel(T)$.  On the other hand, if $m\in T$ and $T'$ is any subset of $\gens(B)\setminus \{m\}$, then $m\notin\Borel(T')$.  
\end{proof}

In light of Proposition \ref{p:uniqueborelgens}, we will often refer to $\Bgens(B)$ as ``the Borel generators'' of $B$.

\begin{proposition} \label{p:identifyborelgens}
Suppose $B$ is a Borel ideal and $m \in B$. Then $m \in \Bgens(B)$ if and only if for all $x_q$ dividing $m$, both $\frac{m}{x_q} \not \in B$ and $m\frac{x_{q+1}}{x_q} \not \in B$.
\end{proposition}

The codimension and projective dimension of a Borel ideal are well-understood, and can be read off from the Borel generators.

\begin{proposition} \label{p:codimpdim}
Let $B$ be Borel with Borel generators $\{m_1,\dots,m_s\}$. Then $\codim B=\max(\min(m_i))$ and $\pd S/B=\max(\max(m_i))$.  If the Borel generators are written in lexicographic order, then $\codim B=\min(m_{s})$.
\end{proposition}
\begin{proof}
By the Eliahou-Kervaire resolution \cite{EK}, the projective dimension of $S/B$ is $\max(\max(m:m\in\gens(B))$.  Let $m$ be a generator of $B$ such that $\max(m)$ is maximal, and $m_{i}$ be such that $m$ precedes $m_{i}$ in the Borel order.  Then $\max(m_{i})=\max(m)$. 

Let $q=\max(\min(m_i))$.  Then, by Borel moves, $B$ contains pure powers of $x_{i}$ for all $i\leq q$; thus $\codim B\geq q$.  On the other hand, every monomial of $B$ is contained in $(x_{1},\dots, x_{q})$, so $\codim B\leq q$.
\end{proof}

All basic operations on Borel ideals can be performed in terms of only their Borel generators.
\begin{proposition} 
Let $B_1=\Borel(T_{1})$ and $B_2=\Borel(T_{2})$ for sets of monomials $T_{1}$ and $T_{2}$.  Then $B_1+B_2=\Borel(T_{1}\cup T_{2})$.
\end{proposition}

\begin{proposition}\label{p:cap}
Let $u=x_{i_{1}}\cdots x_{i_{r}}$ and $v=x_{j_{1}}\cdots x_{j_{s}}$ be written in factored form, and suppose that $r\geq s$.  Put $\ell_{t}=\min(i_{t},j_{t})$ (with $\ell_{t}=i_{t}$ if $s<t\leq r$), and let $w=x_{\ell_{1}}\cdots x_{\ell_{r}}$ be the meet of $u$ and $v$ in the Borel order.  Then $\Borel(u)\cap \Borel(v)=\Borel(w)$.
\end{proposition}

Using the two propositions above, many computations on a Borel ideal 
can be done in terms of principal Borel ideals.  One first works with $\Borel(m)$ for each $m\in\Bgens(B)$, then combines the results (possibly using inclusion-exclusion).  Such an approach allows simpler notation, so we will adopt it without comment wherever possible in the remainder of the paper.

\begin{proposition}\label{p:product}
 $\Borel(u)\cdot \Borel(v)=\Borel(uv)$
\end{proposition}

\begin{corollary}\label{c:up}
 Let $\mathbf{m}=(x_{1},\dots, x_{n})$ be the homogeneous maximal ideal.  Then we have $\mathbf{m}\cdot\Borel(u)= \Borel(ux_{n})$.
\end{corollary}

\begin{proposition} Write $m=x_{i_{1}}\cdots x_{i_{r}}$ in factored form, and let $s$ be minimal such that $i_{s}\geq j$.  Then $(\Borel(m):x_{j})=\Borel(\frac{m}{x_{i_s}})$.  If instead $j\geq i_{r}$, then the colon ideal is $(\Borel(m):x_{j})=\Borel(m)$.  
\end{proposition}

\subsection{Squarefree Borel ideals}

\begin{definition}
We say that an ideal is \emph{squarefree} if it is generated by squarefree monomials and \emph{squarefree Borel} if it is generated by the squarefree monomials of some Borel ideal.  For a set of squarefree monomials $T$, the \emph{squarefree Borel ideal generated by $T$} is the smallest squarefree Borel ideal containing $T$, denoted $\sfborel(T)$.
\end{definition}

Almost all the results on Borel ideals stated above hold (with appropriate modification) for squarefree Borel ideals.  For example, $\sfBorel(T)$ is generated by the squarefree monomials which precede some monomial of $T$ in the Borel order.  The exceptions are Proposition~\ref{p:product}, which doesn't make sense in the squarefree context, and Proposition~\ref{p:codimpdim} and Corollary~\ref{c:up} which become:

\begin{proposition}\label{p:sqcodimpdim}Let $B$ be squarefree Borel with Borel generators $m_{1},\dots, m_{s}$.  The projective dimension of $B$ is $\max(\max(m_{i})-\deg(m_{i}))$.  The codimension of $B$ is $\max(\min(m_{i}))$.  
\end{proposition}
\begin{proof}  A squarefree Borel ideal is resolved by the squarefree part of the Eliahou-Kervaire resolution (\cite{CE, AHH-sqfreelex}), yielding the formula for projective dimension.  For codimension, see \cite[Proposition 4.1]{HS}.
\end{proof}

\begin{proposition}\label{p:squp}
Let $\mathbf{m}=(x_{1},\dots, x_{n})$ be the homogeneous maximal ideal, and let $s$ be maximal such that $x_{s}$ does not divide $u$.  Then the squarefree part of $\mathbf{m}\cdot\sfBorel(u)$ is $\sfBorel(ux_{s})$. 
\end{proposition}

\section{Associated primes} \label{s:primes}
In this section, we compute the associated primes of a Borel ideal.  Our main tools are the following:

\begin{theorem}[Bayer-Stillman] Suppose that $P$ is an associated prime of the Borel ideal $B$.  Then $P=(x_{1},\dots,x_{p})$ for some $p$.
\end{theorem}

\begin{theorem}\label{t:ass}
Suppose that $(x_1,\dots,x_p)$ is an associated prime of the Borel ideal $B$.  Then there exists some Borel generator $m\in\Bgens(B)$ such that $x_p\mid m$ and $\Ann_{S/B}\frac{m}{x_p}=(x_1,\dots,x_p)$.
\end{theorem}

Throughout the section, $B$ is a Borel ideal.

\begin{notation} \label{n.primes}
If $m=x_{i_1} \cdots x_{i_r}$ is a squarefree monomial, write $P_m$ for the prime ideal on the support of $m$, $P_{m}=(x_{i_1},\dots,x_{i_r})$.
If $q\in \mathbb{Z}$, write $P_{q}$ for the prime ideal $(x_{1},\dots,x_{q})$.
\end{notation}

\begin{definition} \label{d.psocle}
$m$ is a \emph{$p$-socle} for $B$ if $\Ann_{S/B}(m) = (x_1, \dots, x_p)$.
\end{definition}

\begin{lemma} \label{l.bgens}
Suppose $m$ is a $p$-socle for $B$. Assume that for all $x_q$ dividing $m$, $m\frac{x_{q+1}}{x_q}$ and $\frac{m}{x_q}$ are not $p$-socles. Then $mx_p \in \Bgens(B)$.
\end{lemma}

\begin{proof}
Because $x_p \in \Ann_{S/B}(m)$, $mx_p\in B$.  Thus, we must show that $\frac{mx_p}{x_{q}}\notin B$ and $mx_{p}\frac{x_{q+1}}{x_{q}}\notin B$ for any $x_{q}$ dividing $mx_{p}$.

First, we show that $m\frac{x_p}{x_q} \not \in B$. Suppose it were in $B$. Then $m \frac{x_p}{\max(mx_p)} \in B$ by the Borel move taking $\max(mx_p)$ to $x_q$. 
If $\max(mx_{p})=x_{p}$, it would follow that $m\in B$, contradicting the fact that $m$ is a $p$-socle.  Hence $\max(mx_{p})=\max(m)$.  
Therefore $\frac{m}{\max(m)}x_p \in B$, meaning $x_p \in \Ann_{S/B}\left(\frac{m}{\max(m)}\right)$. Because $B$ is Borel, this implies that $(x_1,\dots,x_p) \subseteq \Ann_{S/B}\left(\frac{m}{\max(m)}\right)$. Since $\frac{m}{\max(m)}$ is not a $p$-socle, there is a monomial in the variables $x_{p+1}, \dots, x_n$ in $\Ann_{S/B}\left(\frac{m}{\max(m)}\right)$, and hence $\frac{m}{\max(m)}x_{p+1}^N \in B$ for some $N \gg 0$. Multiplying by $\max(m) $ proves that $mx_{p+1}^N \in B$, contradicting the fact that $m$ is a $p$-socle. 

We now prove that $mx_{p}\frac{x_{q+1}}{x_q} \not \in B$. Suppose it were, and let $\mu = \frac{mx_{q+1}}{x_q}$. Then $x_p \in \Ann_{S/B}(\mu)$, and thus $(x_1,\dots,x_p) \subseteq \Ann_{S/B}(\mu)$. By hypothesis, $\mu$ is not a $p$-socle, so $\mu x_{p+1}^N \in B$ for some $N \gg 0$. After a Borel move sending $x_{q+1}$ to $x_q$, we have $\mu x_{p+1}^N \frac{x_q}{x_{q+1}} = mx_{p+1}^N \in B$, contradicting the fact that $m$ is a $p$-socle. Therefore $\frac{mx_{q+1}}{x_q} x_p \not \in B$.
\end{proof}

Theorem \ref{t:ass} follows from Lemma \ref{l.bgens}.

\begin{proof}[Proof of Theorem \ref{t:ass}]
If $(x_{1},\dots, x_{p})\in \Ass(S/B)$, then there is some monomial
$\mu$ which is a $p$-socle.  If there is a variable $x_{q}$ and a
monomial $\mu'$ of the form
$\mu'=\frac{\mu}{x_{q}}$ or $\mu'=\mu\frac{x_{q+1}}{x_{q}}$ which is also a
$p$-socle, replace $\mu$ with $\mu'$.  This process must terminate
since there are finitely many monomials of degree at most
$\deg(\mu)$.  Thus, without loss of generality, the $p$-socle $\mu$
may be chosen to satisfy the hypotheses of Lemma \ref{l.bgens}.  Hence
we may take $m=\mu x_{p}$.  
\end{proof}

Theorem \ref{t:ass} yields the following efficient algorithm for
computing $\Ass(S/B)$ when $B$ is a Borel ideal.

\begin{algorithm}\label{a:ass}  This algorithm computes $\Ass(S/B)$
  for a Borel ideal $B$.
\begin{itemize}
\item[] For each $p$, do the following:
\item[Step 1:]  List all the $m\in\Bgens(B)$ which are divisible by $x_{p}$.
\item[Step 2:]  For each $m$ identified in Step 1, determine if $\frac{m}{x_{p}}$ is a $p$-socle.  If it is, then $(x_{1},\dots, x_{p})$ is associated to $S/B$.  If none are, then $(x_{1},\dots, x_{p})$ is not associated.
\end{itemize}
\end{algorithm}



If $B=\Borel(m)$ is a principal Borel ideal, it is possible to read
off $\Ass(S/B)$ without having to verify any annihilators, as the next
proposition shows.

\begin{proposition}\label{p:principal}
Suppose that $B=\Borel(m)$ and $x_{p}$ divides $m$.  Then
$\frac{m}{x_{p}}$ is a $p$-socle.
\end{proposition}
\begin{proof}
Clearly $x_{p}\in \Ann_{S/B}(\frac{m}{x_{p}})$.  Thus, if $\frac{m}{x_{p}}$
were not a $p$-socle, we would have $\frac{m}{x_{p}}x_{p+1}^{N}\in B$
for large $N$.  But $\frac{m}{x_{p}}x_{p+1}^{N}$ does not precede $m$
in the Borel order for any positive $N$.  
\end{proof}

\begin{corollary}\label{c:principal}
Suppose that $B=\Borel(m)$ is a principal Borel ideal.  Then
$(x_{1},\dots,x_{p})\in \Ass(S/B)$ if and only if $x_{p}$ divides $m$.
\end{corollary}

We thank the referee for suggesting the following corollary.

\begin{corollary}\label{c:anyset}
Let $\mathcal{P}$ be any set of prime ideals all having the form $P_{j}=(x_{1},\dots,x_{j})$.  Then there exists a principal Borel ideal $B$ such that $\mathcal{P}=\Ass(S/B)$. Moreover, $\mathcal{P}=\Ass(S/\Borel(m))$ if and only if $\supp(m) = \{x_{j_1},\dots,x_{j_r}\}$.

\begin{proof}
Let $B=\borel(m)$, where $\displaystyle{m=\prod_{P_{j}\in \mathcal{P}}x_{j}}$.  By Corollary~\ref{c:principal}, $P_{i}\in\Ass(B)$ if and only if $P_{i}\in \mathcal{P}$. The second statement also follows immediately from Corollary~\ref{c:principal}.
\end{proof}
\end{corollary}

We would like to be able to determine $\Ass(S/B)$ for a general Borel ideal $B$ while avoiding socle computations entirely.  The notion of $d$-truncation, defined below, will allow us to do this.

\begin{definition}\label{d:d-trunc}
Let $m$ be a monomial with factorization $x_{i_1} \cdots x_{i_r}$, and let $d$ be a positive integer. If $d\le \deg m$, define the \emph{$d$-truncation} of $m$, denoted $\trunc_{d}(m)$, to be $x_{i_1} \cdots x_{i_d}$. If $d > \deg m$, then set $\trunc_d(m)$ to be $m$ itself. For a monomial ideal $I$, define the \emph{$d$-truncation} of $I$ to be the ideal $\trunc_{d}(I)=(\trunc_{d}(m) : m \in I)$.
\end{definition}

\begin{example}
The $5$-truncation of $m=x_1^2x_3x_4^3x_6x_7$ is $x_1^2x_3x_4^2$, while the $d$-truncation of $m$ is $m$ itself for $d \ge 8$.
\end{example}

\begin{lemma}\label{l:equivtrunc}
Let $B$ be a Borel ideal. Then \[\trunc_{d}(B)
= \Borel(\trunc_{d}(m) : m \in \Bgens(B)).\]  In particular, the $d$-truncation of a Borel ideal is Borel.
\end{lemma} 

\begin{remark}
It is not the case that $\trunc_{d}(I)=(\trunc_{d}(m) : m \in \gens(I))$ for an arbitrary monomial ideal $I$.  For example, let $I=(a^{3},b^{3})$ and put $d=2$.  Then the right-hand side above is $(a^{2},b^{2})$.  However, $ab^{3}\in I$, so the left-hand side contains $ab$.  
\end{remark}

\begin{proof}[Proof of Lemma~\ref{l:equivtrunc}.]  It suffices to show $\trunc_{d}(B)\subseteq \Borel(\trunc_{d}(m) : m \in \Bgens(B))$.  Suppose that $\mu\in \trunc_{d}(B)$.  Then there is some $m\in B$ such that $\mu=\trunc_{d}(m)$.  Thus, there is some Borel generator $w$ for $B$ such that $m$ precedes $w$ in the Borel order.  It follows that $\mu$ precedes $\trunc_{d}(w)$ in the Borel order, so $\mu\in \Borel(\trunc_{d}(m) : m \in \Bgens(B))$ as desired.
\end{proof}

We begin with some observations about how $p$-socles relate to the $d$-truncations of Borel generators.

\begin{lemma}\label{l:claimA}
Let $B$ be a Borel ideal, and suppose $m \in \Bgens(B)$. Let $d$ be the maximal index such that $x_p$ is in the $d^{\text{th}}$ position in $m$. Set $\mu = \frac{\trunc_{d}(m)}{x_p}$. Then $\frac{m}{x_p}$ is a $p$-socle for $B$ if and only if $\mu$ is a $p$-socle for $\trunc_d(B)$.
\end{lemma}

\begin{proof}
Write $B^{*}=\trunc_{d}(B)$.  

First, suppose that $\frac{m}{x_{p}}$ is a $p$-socle for $B$.  If $\mu$ were not a $p$-socle for $B^{*}$, it would follow that $\mu x_{p+1}^{N}\in B^{*}$ for some large $N$.  Since $B^{*}$ is generated in degrees less than or equal to $d$, it follows that $\mu x_{p+1}\in B^{*}$.  Thus there is a Borel generator $\nu^{*}$ of $B^{*}$ which is preceded in the Borel order by $\mu x_{p+1}$, and a Borel generator $\nu$ of $B$ such that $\nu^{*}=\trunc_{d}(\nu)$.  Observe that, for some $q\geq p+1$, $\nu$ has $x_{q}$ in the $d^{\text{th}}$ place; in particular $\frac{\nu}{\nu^{*}}\in k[x_{p+1},\dots,x_{n}]$.  We may multiply $\nu$ by any monomial without leaving $B$, so in particular, it follows that $\nu x_{n}^{N}\in B$ for $N$ greater than the degree of $m$.  

Thus we have \[
\nu x_{n}^{N} =
(\nu^{*}) \left(\frac{\nu}{\nu^{*}}\right) (x_{n}^{N}) \in B.
\]
Applying Borel moves, we have
\[(\mu x_{p+1}) (x_{p+1}^{M})\left(\frac{m}{\mu x_{p}} x_{n}^{N'}\right) \in B,\]
for nonnegative integers $M$ and $N'$.
Rearranging,
\[\frac{m}{\mu x_{p}}\mu x_{p+1} x_{p+1}^{M}x_{n}^{N'}=\frac{m}{x_{p}}x_{p+1}^{M+1}x_{n}^{N'}\in B,\]
contradicting the assumption that $\frac{m}{x_{p}}$ is a $p$-socle for $B$.


Conversely, suppose that $\frac{m}{x_{p}}$ is not a $p$-socle for $B$.  Then, since $m=\frac{m}{x_{p}}x_{p}\in B$, it follows that $\frac{m}{x_{p}}x_{p+1}^{N}\in B$ for some $N$.  In particular, $\trunc_{d}(\frac{m}{x_{p}}x_{p+1}^{N})=\mu x_{p+1}\in B^{*}$, so that $\mu$ is not a $p$-socle for $B^{*}$.

\end{proof}

\begin{lemma}\label{l:lemmaB}
Let $B$ be a Borel ideal, and suppose $m$ is a monomial with $\max(m)=x_p$. If $m \in \Bgens(B)$ and $\deg m \ge \deg m'$ for all $m' \in \Bgens(B)$, then $\frac{m}{x_p}$ is a $p$-socle.
\end{lemma}

\begin{proof}
Note that $\frac{m}{x_p} x_p \in B$, meaning $(x_1,\dots,x_p) \subseteq \Ann_{S/B}(\frac{m}{x_{p}})$. If $\frac{m}{x_p}$ is not a $p$-socle, then $\frac{m}{x_p} x_{p+1}^N \in B$ for some large $N$. Consequently, $\frac{m}{x_p}x_{p+1}^N$ is divisible by some monomial in $B$ of degree $(\deg m)$, which must be $\frac{m}{x_p}x_{p+1}$. Hence $\frac{m}{x_p}x_{p+1} \in B$, but this contradicts $m \in \Bgens(B)$. Therefore $\frac{m}{x_p}$ is a $p$-socle.
\end{proof}

\begin{algorithm}\label{a:nocompute} This algorithm computes $\Ass(S/B)$ for a Borel ideal $B$ without computing socles. Suppose $B$ is generated in degree at most $d$.
\begin{itemize}
\item[Step 1:] For $1 \le i \le d$, compute $\Bgens(\trunc_i(B))$.
\item[Step 2:] For each $m \in \Bgens(\trunc_i(B))$ of degree $i$, $(x_1,\dots,x_{\max(m)}) \in \Ass(S/B)$.
\item[Step 3:] List all primes found in Step 2.
\end{itemize}
\end{algorithm}

The trade-off with Algorithm~\ref{a:ass} is clear. In Algorithm~\ref{a:ass}, for each Borel generator, one has to compute the annihilator of several monomials (or at least compute enough to know whether it is possible for $\frac{m}{x_p}$ to be a $p$-socle). In Algorithm~\ref{a:nocompute}, one needs to compute the Borel generators of all of the $i$-truncations of $B$, but can then read off $\Ass(S/B)$ directly from that list of Borel generators with no further work.

We prove that Algorithm~\ref{a:nocompute} gives a complete list of the primes in $\Ass(S/B)$.

\begin{proof}[Proof of Algorithm~\ref{a:nocompute}]
We begin by proving that the algorithm identifies all elements of $\Ass(S/B)$. We induct on $d$, the maximum degree of a minimal generator of $B$. Let $P_q=(x_1,\dots,x_q)$, and suppose $P_q \in \Ass(S/B)$. If $d=1$, then $B=P_q=\Borel(x_q)$, and the algorithm identifies $P_q$.

If $d > 1$, then there exists $m \in \Bgens(B)$ such that $\frac{m}{x_q}$ is a $q$-socle. Let $e$ be the maximum position of $x_q$ in $m$. If $e=\deg m$, then the algorithm identifies $P_q$ from $m$. If not, then by Lemma~\ref{l:claimA}, $\trunc_{e-1}(m)$ is a $q$-socle for $\trunc_e(B)$. Therefore $P_q$ is associated to $\trunc_e(B)$. Because $e < m \le d$, by induction, the algorithm identifies $P_q$.

Conversely, we need to prove that any prime that Algorithm~\ref{a:nocompute} returns is actually in $\Ass(S/B)$. Suppose there exists a positive integer $e$ such that $m$ is a monomial of degree $e$ in $\Bgens(\trunc_e(B))$, and $\max(m)=q$. Then by Lemma~\ref{l:lemmaB}, $\frac{m}{x_q}$ is a $q$-socle for $\trunc_e(B)$. Let $M \in \Bgens(B)$ be the monomial last in lex order among those monomials whose $e$-truncation is $m$. We have two cases to consider. 

First, suppose $M$ does not have $x_q$ in position $e+1$. Then by Lemma~\ref{l:claimA}, $\frac{M}{x_q}$ is a $q$-socle for $B$. Therefore $P_q \in \Ass(S/B)$. If $M$ does have $x_q$ in position $e+1$, let $f$ be maximal such that $M$ has $x_q$ in position $f$.  Now $\trunc_{f}(M)$ is a minimal Borel generator of $\trunc_{f}(B)$ by the choice of $M$, and, by Lemma~\ref{l:lemmaB}, $mx_{q}^{f-e-1}=\frac{\trunc_{f}(M)}{x_{q}}$ is a $q$-socle for $\trunc_{f}(B)$, and $M$ does not have $x_{q}$ in the $(f+1)^{\text{st}}$ position.  Thus, Lemma~\ref{l:claimA} applies, and $P_{q}\in \Ass(S/B)$ as above.
\end{proof}

\begin{example}\label{e:simplenocompute}
Consider the ideal $B=\Borel(ade,c^4)$ in $k[a,b,c,d,e]$. Then $\trunc_4(B)=\Borel(ade,c^4)$, which shows that $(a,b,c)$ is an associated prime. Moreover, $\trunc_3(B)=\Borel(ade,c^3)$, giving the new prime $(a,b,c,d,e)$. The 2-truncation is $\Borel(ad,c^2)$, so we have another associated prime, $(a,b,c,d)$. Finally, $\trunc_1(B)=\Borel(a,c)=\Borel(c)$, giving no new associated primes because $a$ is not a Borel generator of the 1-truncation. Therefore $\Ass(S/B)=\{(a,b,c),(a,b,c,d),(a,b,c,d,e)\}$.
\end{example}

\subsection{Associated primes for squarefree Borels (Alexander duals)} \label{ss.sqfreeass}

The associated primes of a squarefree ideal correspond to the generators of its Alexander dual, so if $B$ is squarefree Borel, computing $B^{\vee}$ is equivalent to computing $\Ass(S/B)$.  Because $(I+J)^{\vee}=I^{\vee}\cap J^{\vee}$ (since $\gens(I+J)=\gens(I) \cup \gens(J)$), it suffices to compute $B^{\vee}$ for principal squarefree Borel ideals $B$.

\begin{theorem}\label{t:sqfreeass}
Let $m=x_{i_{1}}\dots x_{i_{s}}$ be a squarefree monomial, and $B=\sfborel(m)$.  Then 
\[
B^{\vee}=\sfborel(x_{1}x_{2} \cdots x_{i_{1}}, x_{2}x_{3}\cdots x_{i_{2}},\dots, x_{s}x_{s+1}\cdots x_{i_{s}}).
\]
\end{theorem}

\begin{proof}
Note that $m \in B$ if and only if for all $1 \le j \le s$, $m$ is divisible by at least $j$ variables with index at most $i_j$. Let $C=\sfborel(x_{1}x_{2} \cdots x_{i_{1}}, x_{2}x_{3}\cdots x_{i_{2}},\dots, x_{s}x_{s+1}\cdots x_{i_{s}})$. 

Suppose that $m \in B$ and $\mu$ is a monomial in $C$. We will show that $m \in P_{\mu}$.  
Observe that for some $j$, $\mu$ precedes the monomial $x_{j}\cdots x_{i_{j}}$ in the Borel order.  If on the contrary $m\notin P_{\mu}$, then $m$ would be divisible by 
none of the variables of $P_{\mu}$, so the $j$ variables which divide $m$ and have indices less than or equal to $i_j$ must all be among the variables in the set 
$\{x_{1},\dots, x_{i_{j}}\}\smallsetminus \supp(\mu)$, which has cardinality at most $j-1$.

Conversely, suppose that a monomial $m$ is contained in $P_{\mu}$ for all squarefree $\mu\in C$.  We will show that $m\in B$.  If $m\notin B$, then there is some $j$ for which $m$ is divisible by at most $j-1$ variables with indices less than or equal to $i_{j}$.  Denote these variables by $x_{t_{1}},\dots, x_{t_{r}}$ with $r\leq j-1$.  Put $\nu=\frac{x_{1}\cdots x_{i_{j}}}{x_{t_{1}}\cdots x_{t_{r}}}$.  Then we have $\nu\in C$, but $m\notin P_{\nu}$, a contradiction.
\end{proof}

\begin{remark}  Observe that if $x_{i_{j+1}}=x_{i_{j}+1}$, then $x_{j}\cdots x_{i_{j}}\in \sfborel(x_{j+1}\cdots x_{i_{j+1}})$, so the former generator is redundant.  Thus, in applying Theorem~\ref{t:sqfreeass}, we need only write down the monomials $x_{j}\cdots x_{i_{j}}$ for indices $j$ such that $i_{j}\lneqq i_{j+1}-1$.
\end{remark}

\begin{example}
Let $B=\sfborel(ade,bcf) \subset S=k[a,\dots,f]$. Let $B_1=\sfborel(ade)$ and $B_2=\sfborel(bcf)$. Using Theorem~\ref{t:sqfreeass}, we can compute the Alexander duals of the principal squarefree Borel ideals: $B_1^{\vee} = \sfborel(a,cde)$, and $B_2^{\vee}=\sfborel(bc,cdef)$. Then $B^{\vee} = B_1^{\vee} \cap B_2^{\vee}$. We can now compute $B^{\vee}$ by repeated applications of (the squarefree version of) Proposition~\ref{p:cap}, obtaining
\[
B^{\vee} = \sfborel(ac,adef,bce,cdef)=\sfborel(ac,bce,cdef).
\]
\end{example}

\section{Stanley decompositions} \label{s:stanley}
In this section, we describe a Stanley decomposition for $S/B$, where $B$ is any Borel ideal. Recall that, for a multigraded $S$-module $M$, a \emph{Stanley decomposition} of $M$ is a direct sum decomposition
\[
M=\bigoplus_{f\in \mathcal{I}} f\cdot k[Z_{f}],
\]
where $\mathcal{I}$ is some subset of the multigraded elements of $M$ (in the case of a quotient of $S$ by a monomial ideal, these are monomials) and $Z_{f}$ is a subset of the variables.  See, for example, \cite{notices} for background and connections to the notion of Stanley depth, the subject of much work in combinatorial commutative algebra.  Herzog, Vladoiu, and Zheng investigate Stanley decompositions of Borel ideals in considerably greater generality in \cite{HVZ} than we do here. The primary focus of \cite{HVZ} is computing the Stanley depth of arbitrary monomial modules, i.e., identifying decompositions in which none of the $Z_{f}$ are small. We take a more constructive approach and are interested only in creating an explicit description of one Stanley decomposition of $S/B$, which we will use to study Hilbert functions.  As a bonus, our Stanley decomposition turns out to realize the Stanley depth.

Given a Borel ideal $B$, observe that the truncations of $B$ form a filtration of $S$,
\[
(1)=\trunc_{0}(B)\supseteq \trunc_{1}(B)\supseteq\trunc_{2}(B)\supseteq\dots\supseteq \trunc_{d-1}(B)\supseteq\trunc_{d}(B)=B,
\]
where $d$ is the maximal degree of a generator of $B$.  Taking quotients yields a filtration of $S/B$,
\[
0=S/(\trunc_{0}(B))\subseteq S/(\trunc_{1}(B))\subseteq\dots \subseteq S/(\trunc_{d}(B))=S/B.
\]

Every monomial of $S/B$ occurs in a smallest term in this filtration; we will use this fact to build a Stanley decomposition.

\begin{theorem}\label{t:sdecompone}
Let $B$ be a Borel ideal, generated in degrees less than or equal to $d$.  Then $S/B$ has the Stanley decomposition
\[
S/B=\bigoplus_{s=0}^{d-1}\left(\bigoplus_{\stackrel{m\in \gens(\trunc_{s}(B))\smallsetminus B}{\deg(m)=s}} \hspace{-.25in}m\cdot k[x_{j}:mx_{j}\notin\trunc_{s+1}(B)]\right).
\]
\end{theorem}
\begin{proof}
First, we will show that each of the summands above has trivial intersection with $B$.  Let $m$ be given with $\deg(m)=s$, and write $Z$ for the set of variables $\{x_{j}:mx_{j}\notin\trunc_{s+1}(B)\}$.  Suppose that $\mu\in B\cap \left(m\cdot k[Z]\right)$ is a monomial.  Since $m\notin B$, it follows that $\deg(\mu)\gneqq \deg(m)$.  Hence, $\trunc_{s}(\mu)\in \trunc_{s}(B)$ and $\trunc_{s}(\mu)$ has degree $s$.  We claim that $\trunc_{s}(\mu)=m$.  It suffices to show that every variable of $Z$ has greater index than $x_{\max(m)}$.  Suppose to the contrary that $w\in Z$ has earlier index.  Then, in particular, $mw\notin \trunc_{s+1}(B)$.  On the other hand, there exists some monomial $\tilde{m}\in B$ with $\trunc_{s}(\tilde{m})=m$.  Thus, $\tilde{m}=mm'$ for some monomial $m'$ with $\min(m')\geq \max(m)$. Let $\tilde{w}=\tilde{m}(\frac{w}{x_{\min(m')}})$.  Then $\tilde{w}\in B$, and $\trunc_{s+1}(\tilde{w})=mw\in \trunc_{s+1}(B)$, a contradiction.

Next, we will show that the summands have pairwise trivial intersection.  Let $m\cdot k[Z]$ and $m'\cdot k[Z']$ be two such summands, and suppose that $\mu$ is a monomial in their intersection.  Let $s$ and $s'$ be the degrees of $m$ and $m'$, respectively.  Observe that $\trunc_{s}(\mu)=m$ and $\trunc_{s'}(\mu)=m'$.  Since $m\neq m'$, it follows that $s\neq s'$.  Without loss of generality, we may assume $s<s'$.  Thus, $m$ divides $m'$, so, in particular, $\trunc_{s+1}(m')=\trunc_{s+1}(\mu)$. Let $y=\frac{\trunc_{s+1}(m')}{m}$.  On the one hand, since $my=\trunc_{s+1}(\mu)$, we have $y\in Z$.  On the other hand, since $my=\trunc_{s+1}(m')\in\trunc_{s+1}(B)$, we have $y\notin Z$.

Finally, we will show that every monomial of $S/B$ occurs in one of the summands above.  Let $\mu$ be such a monomial.  Clearly, $\mu\in \trunc_{0}(B)=(1)$.  Thus, there is some maximal $i$ such that $\mu\in \trunc_{i}(B)$.  Let $m=\trunc_{i}(\mu)$.  Then $m$ has degree $i$ and is a generator of $\trunc_{i}(B)$, and $\mu\in m\cdot k[x_{j}:mx_{j}\notin \trunc_{i+1}(B)]$.  
\end{proof}

Recall that the \emph{Stanley depth} of a module is the maximum, over all Stanley decompositions, of the smallest size of any $Z_{f}$ appearing in the decomposition.

\begin{theorem}\label{t:sdepth}
The Stanley decomposition in Theorem~\ref{t:sdecompone} realizes the Stanley depth of $S/B$.
\end{theorem}
\begin{proof} Let $q$ be maximal such that $x_{q}$ divides some Borel generator of $B$.  Observe that the smallest $Z_{f}$ appearing in the decomposition of Theorem~\ref{t:sdecompone} is $Z_{\mu}=\{x_{q+1},\dots,x_{n}\}$, realized when $\mu=\frac{m}{x_{q}}$ for any $m\in\Bgens(B)$ divisible by $x_{q}$.  (This is minimal because every $Z_{f}$ in this decomposition is a terminal sequence of variables.)  Fix one such $\mu$.

Now let $\mathcal{D}$ be another Stanley decomposition of $S/B$.  Then $\mu$ appears in some summand $m^{\ast}k[Z_{m^{\ast}}]$ in $\mathcal{D}$.  Thus, for all $x_{i}\in Z_{m^{\ast}}$, $\mu x_{i}\notin B$.  In particular, $Z_{m^{\ast}}\subseteq Z_{\mu}$.  Thus, $|Z_{\mu}|$ maximizes the minimum dimension of a summand over all Stanley decompositions.
\end{proof}

Stanley \cite{StConj} conjectured that the Stanley depth of a multigraded module is greater than or equal to its depth.  This conjecture is proved for a large class of monomial modules, which includes Borel ideals and quotients by Borel ideals, by Herzog, Vladoiu, and Zheng \cite{HVZ}.  The decomposition above yields a much more direct proof for quotients by Borel ideals, and furthermore shows that equality holds in this case.

\begin{corollary}Stanley's conjecture is sharp for quotients by Borel ideals.
\end{corollary}
\begin{proof} Let $q$ be as in the proof of Theorem~\ref{t:sdepth}.  Then the Stanley depth of $S/B$ is $n-q$, and, by Proposition~\ref{p:codimpdim}, $q$ is the projective dimension of $S/B$. Applying the Auslander-Buchsbaum formula, the depth of $S/B$ is $n-q$.
\end{proof}

We can rewrite the summation of Theorem~\ref{t:sdecompone} in terms of certain quotient modules, without any use of truncation, as follows:
\begin{theorem}\label{t:sdecomptwo}
Let $B$ be Borel.  Then $S/B$ has the Stanley decomposition
\[
S/B=\bigoplus_{j} \left(\bigoplus_{m\in \mathcal{I}_{j}}m\cdot k[x_{j+1},\dots, x_{n}]
\right),
\]
where $\mathcal{I}_{j}$ is the standard monomials of $\displaystyle\frac{B+(x_{j}-1)}{B+(x_{j+1}-1,\dots, x_{n}-1)}$.  The index set $\mathcal{I}_{j}$ is empty unless $\codim B\leq j\leq \pd S/B$.
\end{theorem}

\begin{proof}
We have a filtration
\[
(1)=(B:x_{1}^{\infty})\supseteq (B:x_{2}^{\infty})\supseteq\dots\supseteq (B:x_{n}^{\infty})\supseteq B.
\]
Observe that $m\cdot k[x_{j+1},\dots,x_{n}]$ is a summand in the Stanley decomposition of Theorem~\ref{t:sdecompone} if and only if $m\in (B:x_{j}^{\infty})\smallsetminus (B:x_{j+1}^{\infty})$ and $\max(m)\leq j$, if and only if $m$ is a nonzero monomial of $\frac{(B:x_{j}^{\infty})}{(B:x_{j+1}^{\infty})}\cong\frac{B+(x_{j}-1,\dots, x_{n}-1)}{B+(x_{j+1}-1,\dots, x_{n}-1)}$ and $\max(m)\leq j$.  Every standard monomial modulo $(B+(x_{j+1}-1,\dots, x_{n}-1))$ has max at most $j$, so the summands of the form $m\cdot k[x_{j+1},\dots,x_{n}]$ are in one-to-one correspondence with the standard monomials of the Artinian module $\frac{B+(x_{j}-1,\dots, x_{n}-1)}{B+(x_{j+1}-1,\dots, x_{n}-1)}$.
\end{proof}

Since each summand of the form $m\cdot k[Z_{m}]$ contributes a summand of $\frac{t^{\deg(m)}}{(1-t)^{|Z_{m}|}}$ to the Hilbert series, Theorem~\ref{t:sdecomptwo} yields the following formula for the Hilbert series of $S/B$:
\[
\text{HS}(S/B)=\sum_{j}\frac{\text{HS}\left(\frac{B+(x_{j}-1)}{B+(x_{j+1}-1,\dots, x_{n}-1)}\right)}{(1-t)^{n-j}}
\]
This formula lacks aesthetic appeal, but it allows us to compute the multiplicity of $S/B$.
\begin{proposition}
Write $B=\borel(m_{1},\dots, m_{r})$, 
and let $p=\codim(B)$.  Then the multiplicity of $S/B$ is
\[
e(S/B)=\dim_{k}\left(\frac{S}{(B,x_{p+1}-1, \dots, x_{n}-1)}\right).
\]
\end{proposition}

\begin{proof}
The first nonzero summand is $\displaystyle\frac{\text{HS}\left(\frac{S}{B+(x_{p+1}-1,\dots,x_{n}-1)}\right)}{(1-t)^{n-p}}$ because  $B+(x_{j}-1)=(1)$ when $j\leq p$. All other summands have lower powers of $(1-t)$ in the denominator.  Evaluating the $h$-polynomial at $t=1$ thus gives us $\dim_{k}(S/(B,x_{p+1}-1, \dots, x_{n}-1))$.
 \end{proof}

In the case that $B=\borel(m)$ is a principal Borel ideal, Theorem~\ref{t:sdecompone} yields a more elegant expression for the Hilbert series of $S/B$.
Factor $m=x_{i_1}x_{i_2} x_{i_3} \cdots x_{i_s}$.
Then, if $u$ is a monomial generator of $\trunc_{d}(B)$, we compute $(\trunc_{d+1}(B):u)=(x_{1},\dots, x_{i_{d+1}})$.  Thus, in particular, $Z_{u}=\{x_{1+i_{d+1}},\dots, x_{n}\}$. This proves the following:

\begin{proposition}\label{c:princstanley}
Let $B=\borel(m)$ be a principal Borel ideal, and factor $m=x_{i_{1}}\dots x_{i_{s}}$.  Then
\[
\HS(S/B) = \sum_{d=0}^{s-1} \frac{c_{d}t^d}{(1-t)^{n-i_{d+1}}},
\]
where $c_{d}$ is the number of degree $d$ generators of $\trunc_{d}(B)$.  The multiplicity of $S/B$ is $e(S/B)=\displaystyle\sum_{i_{d+1}=i_{1}}c_{d}$.
\end{proposition}

In the next section, we will compute the $c_{d}$ using generalized Catalan triangles.

\section{Catalan diagrams}\label{s:catalan}

We begin by recalling some useful notation in the study of Hilbert functions and Betti numbers of Borel ideals.

\begin{definition}
Let $B$ be a Borel (or squarefree Borel) ideal.  Write $B=(m_{1},\dots, m_{s})$ in terms of its minimal monomial generators.  For each $i$, $1\leq i\leq n$, and each degree $d$ let $w_{i}^{d}(B)$ be the cardinality of the set $\{m\in B:\max(m)=i \text{ and } \deg(m)=d\}$ and put 
\[
w_{\leq i}^{d}(B)=\displaystyle\sum_{j\leq i}w_{j}^{d}(B).
\]  
If $B$ is a (principal) Borel ideal generated in degree $d$, we simply write $w_{i}(B)$ instead of $w_{i}^{d}(B)$.
\end{definition}

The numbers $w_{i}^{d}(B)$ were introduced by Bigatti in \cite{Bigatti} and have been used to study Hilbert functions and Betti numbers of Borel ideals, particularly in comparison to lex ideals.  See \cite{Bigatti, Fr, Green, Mermin-Peeva}.

In the formula in Proposition~\ref{c:princstanley}, the number $c_{d}$ is $w_{\leq i_{d}}(\trunc_{d}(B))$.


The remainder of this section deals with the computation of $w_{i}(B)$ for a principal Borel ideal $B$.



\begin{definition}
For a monomial $m$ with factorization $x_{i_1}x_{i_2} \cdots x_{i_r}$, define the \emph{Catalan diagram with shape $m$}, written $C(m)$, as follows.  First construct a left-justified array of boxes whose $j^{th}$ row (from the top) has $i_j$ boxes.  Then put a $1$ in each box along the top row.  Finally, fill in the remaining boxes by the following rule: the entry in the $k^{th}$ box of the $j^{th}$ row is the sum of the entries in the first $k$ boxes of the $(j-1)^{st}$ row. (If the $(j-1)^{st}$ row has fewer than $k$ boxes, simply fill in the box with the sum of all the entries in the $(j-1)^{st}$ row.) Note that we number rows from top to bottom and columns from left to right, starting with 1 in both cases.
\end{definition}

\begin{figure}[htp]
\centering
\includegraphics[scale=.60]{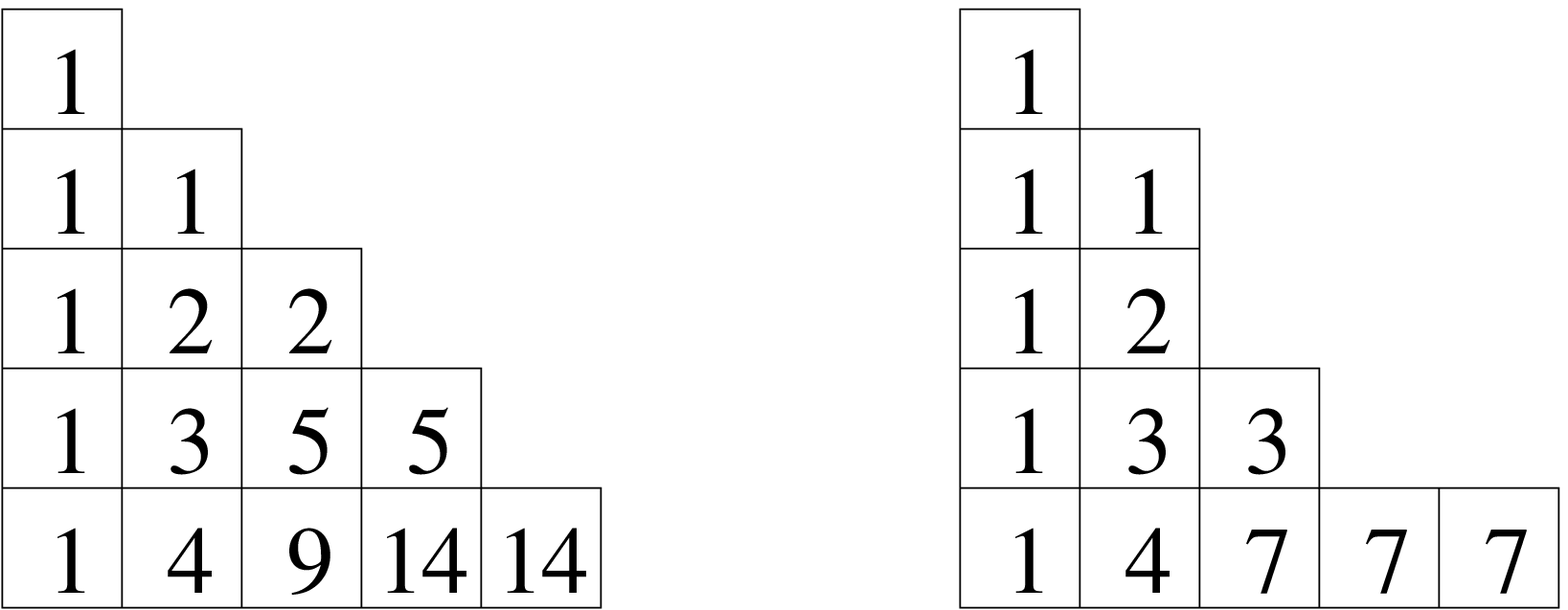}
\caption{The Catalan diagrams $C(x_1x_2x_3x_4x_5)$ and $C(x_1x_2^2x_3x_5)$, respectively.}\label{catalandiagrams}
\end{figure}

\begin{example}\label{e:triangle}
If $m=x_{1}x_{2}\cdots x_{d}$, then the Catalan diagram with shape $m$ is the first $d$ rows of Catalan's triangle (see Figure \ref{catalandiagrams}).  The entries on its main diagonal are the Catalan numbers.  These are central objects in enumerative combinatorics; for background see \cite{ECII} or \cite{Koshy}.
\end{example}


If $B=\Borel(m)$, we can read off the numbers $w_{i}(B)$ from $C(m)$, the
Catalan diagram with shape $m$.

\begin{theorem}
The entry in the $j^\text{th}$ box of the bottom row of $C(m)$ is $w_j(\borel(m))$.
\end{theorem}

\begin{proof}
Let $d=\deg(m)$. 
We induct on $d$.  If $d = 1$, then $m = x_{i_1}$, so $B = (x_1, x_2, \ldots, x_{i_1})$ and $w_j(B) = 1$ for $j \leq i_1$.  Now suppose $d > 1$, $k \leq i_d$, and consider the map from $\{\mu \in \gens(B) : \max(\mu) = x_k\}$ to $\{\mu' \in \gens(\trunc_{d-1}(B)) : \max(\mu') \in \{x_1, x_2, \ldots, x_k\}\}$ defined by $\mu \mapsto \frac{\mu}{x_k}$.  This map is clearly injective.  By induction, the cardinality of the latter set is the sum of the entries in the first $k$ boxes of the $(d-2)^\text{th}$ row of $C(m)$.  To see that this map is surjective, let $\mu'$ be a monomial of $\trunc_{d-1}(B) = \Borel(x_{i_1} x_{i_2} \cdots x_{i_{d-1}})$ with degree $d-1$.  Then $\mu' \bgeq x_{i_1}x_{i_2} \cdots x_{i_{d-1}}$, so $\mu' x_k \bgeq x_{i_1}x_{i_2} \cdots x_{i_d} = m$ because $k \leq i_d$.  So $\mu'x_k \in B$, and the claim follows.
\end{proof}

Implicit in this construction is a bijection between generators of a principal Borel ideal and  planar lattice paths within the diagram $C(m)$.  This is a special case of a \emph{lattice path polymatroid}; see \cite{Schweig}.  

\begin{example}\label{e:artinian}
Consider $B=\mathbf{m}^{d}=\borel(x_{n}^{d})$.  Then the Catalan diagram with shape $m$ is a rectangle with height $d$ and width $n$. This is Pascal's triangle, with the $i^{\text{th}}$ row of Pascal's triangle along the $i^{\text{th}}$ antidiagonal. The entry in the $k^{th}$ box of the $j^{th}$ row is ${j +k -2 \choose j-1}$.  In particular, $w_{i}(B)=\binom{d+i-2}{d-1}$.
\end{example}

\begin{figure}[htp]
\centering
\includegraphics[scale=.60]{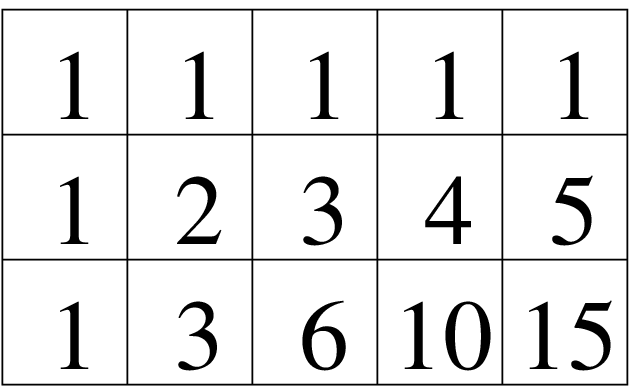}
\caption{The Catalan diagram $C(x_5^3)$.}\label{rectangle}
\end{figure}

\begin{example} \label{e:hs}
Let $B =  \borel(x_{1}x_{2}x_{3}x_{4}x_{5})$.  We compute the Hilbert series of the quotient $S/B$.  The Catalan diagram $C(x_{1}x_{2}x_{3}x_{4}x_{5})$ is computed in Figure \ref{catalandiagrams}.  For $d>0$, the number of generators of $\trunc_{d}(B)$ is the sum of the entries in the $d^{\text{th}}$ row of this diagram.  (Recall that $\trunc_{0}(B)=1$ has a single generator.)  Thus we compute $c_{0}=1, c_{1}=1, c_{2}=2, c_{3}=5$, and $c_{4}=14$ (the first five Catalan numbers).  By Corollary~\ref{c:princstanley}, the Hilbert series of $S/B$ is 
\[
\frac{1}{(1-t)^{4}} + \frac{t}{(1-t)^{3}} + \frac{2t^{2}}{(1-t)^{2}} + \frac{5t^{3}}{1-t} + 14t^{4}.
\]
$S/B$ has multiplicity one (the coefficient on $\frac{1}{(1-t)^{4}}$) and $h$-polynomial $1+t+t^{2}+t^{3}+t^{4}-41t^{5}+79t^{6}-56t^{7}+14t^{8}$.  Since the resolution is linear, we can read off the Betti numbers from $(1-t)h(t)$.
\end{example}

\begin{example}
Suppose that $B=\Borel(m)$ is a lexicographic ideal.  Then, by \cite[Prop. 3.2(i)]{Bo}, $m$ has the form $x_{1}^{a}x_{j}x_{n}^{b}$.  A straightforward lattice-path argument on 
$C(m)$ shows that $w_{i}(B)=\displaystyle\sum_{\ell=1}^{j}\textstyle \binom{b-1+i-\ell}{i-\ell}$.  Some manipulation yields 
$|B_{\deg(m)}|=\displaystyle\sum_{\ell=1}^{j}\textstyle \binom{b+n-\ell}{n-\ell}$.  In particular, the $(n-1)^{\text{st}}$ Macaulay representation of $|B|$, which governs the growth of the Hilbert function of $B$, has the property that the numerators decrease by exactly one between consecutive summands. 
\end{example}

Next, we relate Catalan diagrams to squarefree Borel ideals.

Let $m = x_{i_1}x_{i_2} \cdots x_{i_d}$ be a squarefree monomial, with $i_1 < i_2 < \cdots < i_d$.  Define a (not necessarily squarefree) monomial $\tau(m)$ by 
\[
\tau(m) = x_{i_1}x_{i_2 -1} x_{i_3 -2 } \cdots x_{i_d - d +1}.
\]
Note that $\tau$ is the inverse to the map $\sigma$ defined in \cite{AHHshiftingops}.

\begin{proposition}[\cite{AHHshiftingops}]
The map $\tau$ is a bijection between the squarefree monomials and the monomials of $k[x_{1},x_{2},\dots,x_{\infty}]$.
\end{proposition}

\begin{corollary}
If $m$ is a squarefree monomial, $\tau$ is a bijection between $\gens ( \sfBorel(m))$ and $\gens(\Borel(\tau(m)))$.
\end{corollary}

\begin{corollary}\label{tau}
Let $m$ be a squarefree monomial of degree $d$, and put $I = \sfBorel(m)$ and $J = \Borel(\tau(m))$.  Then $w_{i}(I) = w_{i-d+1}(J)$ for all $i$. 
\end{corollary}

\begin{proof}
For any monomial $\mu$ of degree $d$, $\max(\mu) = x_i$ if and only if $\max(\tau(\mu)) = x_{i-d+1}$.  
\end{proof}

Now let $B = \sfBorel(m)$ for some squarefree monomial $m$.  By Corollary \ref{tau}, we can compute the numbers $w_{i}(B)$: $w_{i}(B)$ is the $(i-d+1)^{\text{st}}$ entry in the bottom row of $C(\tau(m))$.

\section{Betti numbers}\label{s:Betti}

It is well known that the Betti numbers of a Borel ideal depend only on the $w_i$.  (See, for example, \cite{Bigatti}.)  We recall the Eliahou-Kervaire resolution \cite{EK}, which makes the necessary computation explicit.

\begin{construction}
Let $B$ be a Borel ideal, and let $m$ be a minimal monomial generator of $B$ with $\max(m)=x_j$.  Let $\alpha$ be any squarefree monomial in $k[x_{1},\dots, x_{j-1}]$.  The formal symbol $(m,\alpha)$ is called an \emph{Eliahou-Kervaire symbol} or \emph{EK symbol} for $B$.  We assign the EK symbol $(m,\alpha)$ homological degree equal to the degree of $\alpha$ and multidegree equal to $m\alpha$.
\end{construction}

\begin{theorem}[\cite{EK}]
The Eliahou-Kervaire symbols form a basis for a minimal free resolution of the ideal $B$.
\end{theorem}

A nice exposition of the Eliahou-Kervaire resolution in terms of mapping cones is \cite{PS}.  This is one of the only known explicit resolutions, so it has been heavily studied; see, for example, \cite{BW, CE, Cl, GHP, Me}.

\begin{corollary} The graded Betti number $b_{i,j}(B)$ is the number of Eliahou-Kervaire symbols with homological degree $i$ and internal degree $j$.  The graded Betti number $b_{i,j}(S/B)$ is 1 if $i=j=0$ and $b_{i-1,j}(B)$ otherwise.
\end{corollary}

For a given monomial generator $m$, the number of $\alpha$ such that $(m,\alpha)$ is an EK symbol of homological degree $i$ is $\binom{\max(m)-1}{\alpha}$.  Summing over all $m$ gives us a formula for the Betti numbers in terms of the $w_i$:

\begin{proposition}\label{p:betti} Suppose that $B$ is generated entirely in degree $d$.  Then 
\[
b_i(B)=b_{i,i+d}(B)=\displaystyle\sum_{j=1}^{n}\binom{j-1}{i}w_{j}(B).
\]  
The graded Poincar\'e series for $B$ is 
\[
P_{B}(t,u)=u^{d}\displaystyle\sum_{i=1}^{n}w_{i}(B)(1+tu)^{i-1}.
\]
\end{proposition}

\begin{example} 
Suppose $B=\Borel(x_{n}^{d})$ is the $d^{\text{th}}$ power of the maximal ideal.  Then $w_{i}(B)=\binom{i+d-2}{d-1}$, and $b_{i,i+d}(B)=\displaystyle\sum_{j=1}^{d}\textstyle\binom{j+d-2}{d-1}\binom{j-1}{i}$.  This has been known since before the Eliahou-Kervaire resolution was discovered; see for example \cite{BR}.
\end{example}

\begin{corollary}
Suppose that $B=\borel(m)$ is a principal Borel ideal.  Let $C(m)$ be the Catalan diagram of shape $m$, and let $g(t)$ be the generating function with coefficients given by the bottom row of $C(m)$.  Then $tg(t)$ is the generating function for $w_{i}(B)$, and the Poincar\'e series for $B$ is $u^{\deg(m)}g(1+tu)$.
\end{corollary}

Proposition~\ref{p:betti} enables us to compute the Betti numbers for an arbitrary Borel ideal, one Borel generator at a time:

\begin{lemma}\label{l:oneatatime}
Suppose that $B=\Borel(m_{1},\dots, m_{r})$, with $d=\deg(m_{r})\geq \deg(m_{i})$ for all $i$.  Put $B'=\borel(m_{1},\dots, m_{r-1})$ and $B''=\borel(m_{r})\cap B'$.  Then $B''$ is generated entirely in degree $d$ and $b_{i,j}(B)= b_{i,j}(B')+b_{i,j}(\borel(m_{r}))-b_{i,j}(B'')$.   
\end{lemma}

\begin{example}
We compute the Betti numbers of $B=\Borel(a,b^2,c^3)$, a special case of \cite[Example 2]{EK}, in which Eliahou and Kervaire point out that Catalan numbers arise in computing total Betti numbers of $\Borel(x_1,x_2^2,\dots,x_n^n)$. By Proposition~\ref{p:cap}, $\Borel(a) \cap \Borel(b^2)=\Borel(ab)$. Applying Lemma~\ref{l:oneatatime}, we have
\[
b_{i,j}(\Borel(a,b^2)) = b_{i,j}(\Borel(a)) + b_{i,j}(\Borel(b^2)) - b_{i,j}(\Borel(ab)).
\]
The only row in the Catalan diagram of shape $a$ is (1), the bottom row of the Catalan diagram of shape $b^2$ is $(1, 2)$, and for $ab$, it is $(1,1)$. Multiplying by the appropriate binomial coefficients, we calculate that $b_{0,1}$, $b_{0,2}$, and $b_{1,3}$ of $\Borel(a,b^2)$ are one, and the other graded Betti numbers are zero.

Now we apply Lemma~\ref{l:oneatatime} again to get the graded Betti numbers of $B$ itself. By Proposition~\ref{p:cap}, $\Borel(a,b^2) \cap \Borel(c^3) = \Borel(ac^2,b^2c)$. Thus 
\[
b_{i,j}(B) = b_{i,j}(\Borel(a,b^2)) + b_{i,j}(\Borel(c^3)) - b_{i,j}(\Borel(ac^2,b^2c)).
\]

Let $B_1=\Borel(a,b^2)$, $B_2=\Borel(c^3)$, and $B_3=\Borel(ac^2,b^2c)$. To get the graded Betti numbers of $B_3$, we compute $\Borel(ac^2) \cap \Borel(b^2c) = \Borel(abc)$ and use Lemma~\ref{l:oneatatime}. The bottom row of the Catalan diagram of shape $ac^2$ is $(1,2,3)$; for $b^2c$, the bottom row is $(1,3,3)$, and for $abc$, the bottom row is $(1,2,2)$. To compute the graded Betti numbers of $\Borel(c^3)$, note that the bottom row of the Catalan diagram of shape $c^3$ is $(1, 3, 6)$. Thus, after multiplying by the binomial coefficients, we have the following graded Betti diagrams, switching to the quotients to use the standard Macaulay 2 output:

\medskip
\begin{center}
\begin{tabular}{ccccccccccccccccccc} 
$S/B_1$: & total: & 1 & 2 & 1 & & $S/B_2$: & total: & 1 & 10 & 15 & 6 & & $S/B_3$: & total: & 1 & 8 & 11 & 4\\
& 0: & 1 & 1 & . & & & 0: & 1 & . & . & . & & & 0: & 1 & . & . & .\\
& 1: & . & 1 & 1 & & & 1: & . & . & . & . & & & 1: & . & . & . & .\\
& & & & & & & 2: & . & 10 & 15 & 6 & & &2: & . & 8 & 11 & 4
\end{tabular}
\end{center}

Adding the first two diagrams and subtracting the third, we obtain the graded Betti diagram of $S/B$:

\medskip

\begin{center}
\begin{tabular}{cccccc} 
$S/B$: & total: & 1 & 4 & 5 & 2 \\
& 0: & 1 & 1 & . & . \\
& 1: & . & 1 & 1 & . \\
& 2: & . & 2 & 4 & 2 
\end{tabular}
\end{center}
\end{example}

\bigskip

In a different direction, our techniques also allow us to consider the resolution of the residue field $k$ over $S/B$, where $B$ is a Borel ideal. We may assume that $B$ is generated by monomials of degree at least two. Then by \cite[Corollary 1.2]{Peeva-dgborel}, $S/B$ is Golod, and hence the graded Poincar\'e series of $k$ over $S/B$ can be expressed as 
\[ 
P^{S/B}_k(t,u) = \frac{(1+tu)^n}{1-t^{2}P_{B}^{S}(t,u)}.
\]

\begin{proposition}\label{p.kresoverb}
Let $B$ be a Borel ideal generated in a single degree $d$, and let $f(t)$ be the generating function on the $w_i(B)$. Then 
\[
P^{S/B}_k(t,u)= \frac{(1+tu)^{n+1}}{1+tu-t^2 u^d f(1+tu)}.
\]
If $B=\borel(m)$ is principal, let $g(t)=\frac{1}{t}f(t)$ be the generating function on the bottom row of $C(m)$.  Then
\[
P^{S/B}_k(t,u) = \frac{(1+tu)^{n}}{1-t^2 u^d g(1+tu)}.
\]

\end{proposition}


\begin{proof}
Because $w_i(B)$ counts the number of elements $m$ of $\gens(B)$ with $\max(m)=i$, we have
\begin{align*}
P^{S/B}_k(t,u) 
               &= \frac{(1+tu)^n}{1-t^2u^d\displaystyle\sum_{i=1}^n w_i(B)(1+tu)^{i-1}}\\
			   &= \frac{(1+tu)^n}{1-t^2u^d\dfrac{f(1+tu)}{1+tu}}\\
               &= \frac{(1+tu)^{n+1}}{1+tu-t^2u^df(1+tu)}.\qedhere
\end{align*}
\end{proof}

\subsection{Betti numbers for squarefree Borel ideals}

Now suppose that $B$ is a squarefree Borel ideal in $S$. Then $B$ is resolved by the squarefree part of the Eliahou-Kervaire resolution \cite{CE,AHH-sqfreelex}; that is, the basis for the resolution is given by EK symbols with squarefree multidegree. Thus the Poincar\'e series of $B$ is given by 
\[
P^{S}_{B}(t,u)=\sum_{m\in\gens(B)}u^{\deg(m)}(1+tu)^{\max(m)-\deg(m)}.
\]

For the rest of the section, assume that $B$ is a squarefree Borel ideal generated in a single degree $d$, and let $f(t)$ be the generating function on the $w_{i}(B)$.  If $B=\sfborel(m)$ is  principal squarefree Borel, let $g(t)$ be the generating function on the last row of the Catalan diagram $C(\tau(m))$. Recall that $f(t)=t^{d}g(t)$.

\begin{proposition}
The graded Poincar\'e series of $B$ over $S$ is 
\[
P^{S}_{B}(t,u)=f(1+tu)\left(\frac{u}{1+tu}\right)^{d}.
\]
If $B$ is principal squarefree Borel, then 
\[
P^{S}_{B}(t,u)=u^{d}g(1+tu).
\]
\end{proposition}
\begin{proof}
We have
\begin{align*}
P^{S}_{B}(t,u)&=u^{d}\sum_{m\in\gens(B)}(1+tu)^{\max(m)-\deg(m)}\\
&=u^{d}\sum_{i=d}^{n}w_{i}(B)(1+tu)^{i-d}\\
&=u^{d}\frac{f(1+tu)}{(1+tu)^{d}}.\qedhere
\end{align*}
\end{proof}

By \cite{AHH-sqfreelex}, the ring $S/B$ is Golod.  Thus an argument analogous to that of Proposition~\ref{p.kresoverb} gives the following formula for the resolution of $k$ over $S/B$:
\begin{proposition}
The Poincar\'e series for $k$ over $S/B$ is
\[
P_{k}^{S/B}(t,u)=\frac{(1+tu)^{n+d}}{(1+tu)^{d}-t^{2}u^{d}f(1+tu)}.
\]
If $B$ is principal squarefree Borel, the graded Poincar\'e series is
\[
P_{k}^{S/B}(t,u)=\frac{(1+tu)^{n}}{1-t^{2}u^{d}g(1+tu)}.
\]
\end{proposition}

Aramova, Herzog, and Hibi \cite{AHH-gotzmann} compute the minimal free resolution of a squarefree Borel ideal in the exterior algebra $E$.  In \cite[Corollary 3.3]{AHH-gotzmann}, they give the graded Poincar\'e series:
\[
P_{B}^{E}(t,u)=\sum_{m\in \gens(B)}\frac{u^{\deg(m)}}{(1-tu)^{\max(m)}}.
\]

Let $P=(x_{1}^{2},\dots,x_{n}^{2})$.  Then a squarefree Borel ideal is naturally defined over $S/P$ in the same way that it is over $E$.  The same mapping cone argument minimally resolves $B$ over both $E$ and $S/P$, so the Poincar\'e series is the same in both cases.
When $B$ is generated in a single degree, this simplifies using the $w_{i}$:

\begin{proposition}
The graded Poincar\'e series of $B$ over $E$ or $S/P$ is
\[
P_{B}^{E}(t,u)=P_{B}^{S/P}(t,u)=
u^{d}f\left(\frac{1}{1-tu}\right).
\]
If $B$ is principal squarefree Borel, then 
\[
P_{B}^{E}(t,u)=P_{B}^{S/P}(t,u)=\left(\frac{u}{1-tu}\right)^{d}g\left(\frac{1}{1-tu}\right).
\]

\end{proposition}

\section{Pointed pseudo-triangulations} \label{s:ppt}

In this section, we uncover a surprising connection between the Betti numbers of some principal Borel ideals and the number of pointed pseudo-triangulations of certain point sets in the plane. Fix a monomial $m$ with $\max(m) = x_k$, and consider the principal Borel ideal $B = \Borel(m)$. We begin by determining a simple expression, not requiring a sum, for the Betti numbers of particular principal Borel ideals. Since principal Borel ideals have linear resolutions, we suppress the degree and write total Betti numbers throughout.


\begin{proposition}\label{p:recursion}
For each $i$, we have 
\[
b_{i-1}(B) + b_{i}\left(\Borel\left(\frac{m}{x_k}\right)\right) = {k \choose i} w_k(B).
\]
\end{proposition}

\begin{proof}
Consider the set of pairs 
\[
X = \{(\mu, \alpha): \mu\in \gens(B), \max(\mu) = x_k, \max(\alpha) \leq k, \deg(\alpha) = i\}
\]
where each $\alpha$ is a squarefree monomial, and note that $|X| = {k \choose i} w_k(B)$.  We mention that pairs in $X$ are not necessarily EK symbols, as we allow $x_k$ to occur in $\alpha$. 

Now write $X$ as the disjoint union $X = X_1 \sqcup X_2$, where $X_1 = \{(\mu, \alpha) \in X : \max(\alpha) < \max(\frac{\mu}{x_k})\}$ and $X_2 = X\setminus X_1$.  

When $\max(\mu) = x_k$, $\frac{\mu}{x_k} \in \gens(\Borel(\frac{m}{x_k}))$ if and only if $\mu \in \gens(\Borel(m))$.  Therefore, the map 
\[
(\mu, \alpha) \mapsto \left(\frac{\mu}{x_k}, \alpha\right)
\]
gives a bijection between $X_1$ and EK symbols $(\mu', \alpha)$ of $\Borel(\frac{m}{x_k})$ such that $\deg(\alpha) = i$.  Thus $|X_1| = b_{i}(\Borel(\frac{m}{x_k}))$.  

Now let $(\mu, \alpha)$ be an EK symbol of $B$ with $\deg(\alpha) = i -1$, and consider the map
\[
(\mu, \alpha) \mapsto \left(\frac{\mu x_k}{\max(\mu)}, \max(\mu)\alpha\right).
\]
We claim this map gives a bijection between such EK symbols and pairs in $X_2$.  Clearly, $\frac{\mu x_k}{\max(\mu)}$ is a generator of $B$ counted by $w_k(B)$.  Now let $(\mu', \alpha') \in X_2$, let $\max(\frac{\mu'}{x_k}) = x_j$, and let $\max(\alpha') = x_q$.  Since $(\mu', \alpha') \in X_2$, $q \geq j$, $\frac{x_q\mu'}{x_k}$ is a generator of $B$, meaning the map
\[
(\mu', \alpha') \mapsto \left(\frac{x_q\mu'}{x_k}, \frac{\alpha'}{x_q}\right)
\]
is an inverse to the above map.  Therefore, $|X_2| = b_{i-1}(B)$, and the result follows.
\end{proof}

\begin{example} \label{e:maxpowers}
When $B=\Borel(x_n^d)$, a power of the maximal ideal, the recursion gives 
\[ 
b_{i-1}(\Borel(x_n^d)) + b_i(\Borel(x_n^{d-1}) = \binom{n}{i} w_n(\Borel(x_n^d)) = \frac{n}{d+n-1} \binom{d+n-1}{i,n-i,d-1}.
\]
\end{example}

Recall from Example \ref{e:triangle} that $w_n(\Borel(x_1x_2 \cdots x_n)) = C_{n-1}$, the $(n-1)^{\text{st}}$ Catalan number.  Thus this ideal is interesting from a combinatorial perspective.  The following is a special case of Proposition~\ref{p:recursion}.

\begin{corollary}\label{c:firstvars}
For any $n$, the following recursion holds: 
\[
b_{i-1}(\Borel(x_1x_2 \cdots x_n)) + b_{i}(\Borel(x_1x_2 \cdots x_{n-1})) = {n \choose i } C_{n-1}.
\]
\end{corollary}

\begin{proposition}\label{c:closedform}
For all $n \ge 1$ and all $i$, \[ b_{i}(\Borel(x_1 \cdots x_n)) = \frac{1}{n}\binom{2n}{n-i-1} \binom{n+i-1}{i}.\]
\end{proposition}

\begin{remark}\label{r:oeis}
This is an alternate formula for the unsigned version of sequence A062991 in the Online Encyclopedia of Integer Sequences \cite{oeis} (with the indexing shifted). It is listed in a comment in the entry in the OEIS, but no proof is given, so we sketch the argument below.
\end{remark}

\begin{proof}
We outline the computational proof, inducting on $n$. When $n=1$, the $0^{\text{th}}$ Betti number is one, and all other Betti numbers are zero, consistent with the formula. Assume now that the formula is true for some integer $n-1$, where $n \ge 2$, for all values of $i$. By Corollary~\ref{c:firstvars},
\[
b_i(\Borel(x_1 \cdots x_n)) = \binom{n}{i+1}C_{n-1} - b_{i+1}(\Borel(x_1\cdots x_{n-1})).
\]
Using the inductive hypothesis, we have
\begin{align*}
b_i(\Borel(x_1 \cdots x_n)) & = \frac{1}{n}\binom{n}{i+1}\binom{2n-2}{n-1} -  \frac{1}{n-1}\binom{2n-2}{n-i-3}\binom{n+i-1}{i+1}\\
& = \frac{1}{n}\binom{2n}{n-i-1}\binom{n+i-1}{i}\left(\frac{4in-2i+4n-2}{(i+1)(2n-1)(2)}\right)\\
& = \frac{1}{n}\binom{2n}{n-i-1}\binom{n+i-1}{i}.
\end{align*}
The intermediate computations simply involve factoring out $\frac{1}{n}$ and the appropriate binomial coefficients; then one checks that what remains is equal to one.
\end{proof}

The (unsigned version of) sequence A062991 in \cite{oeis} that gives Betti numbers of the ideals $\Borel(x_1 \cdots x_n)$ also arises in the paper \cite{AOSS}, which is devoted to counting pseudo-triangulations of particular point sets. For full details on the combinatorial terminology, see \cite{AOSS}. Briefly, a \emph{pseudo-triangle} is a polygon in the plane with exactly three convex vertices with internal angles measuring less than $\pi$. A \emph{pseudo-triangulation} is a partition of the convex hull of a point set $A$ in the plane into pseudo-triangles whose vertex set is exactly $A$. We say that a pseudo-triangulation is \emph{pointed} if every vertex has an incident angle greater than $\pi$. Pointed pseudo-triangulations have arisen in a number of settings recently; see, for example, \cite{RSS}. 
A set $A$ of $\ell+3$ points in the plane a \emph{single chain} if $A$ consists of $\ell+2$ points labeled $\alpha, p_1,\dots, p_{\ell}, \beta$ that form a convex $(\ell+2)$-gon and one additional point labeled $q$, outside the $(\ell+2)$-gon, but from which all edges but $(\alpha,\beta)$ are visible. Let $a(\ell,i)$ be the number of pointed pseudo-triangulations of $A$ in which the point $q$ is connected to exactly $i$ of the points $p_{j}$. (By \cite[Theorem 6]{AOSS}, these numbers also count certain triangulations of particular convex $(\ell+3)$-polygons.)

\begin{theorem}\label{e:ppt}
Let $a(\ell,i)$ be as above.  Then $a(\ell, i)=b_{\ell-i}(\borel(x_{1}\cdots x_{\ell+1}))$.
\end{theorem}

\begin{proof}
It is enough to show that the Betti numbers satisfy the same recursion as the $a(\ell,i)$ after the shifting of indices. The recursion for the $a(\ell,i)$ is \cite[Theorem 14]{AOSS}, which states that
\[
a(\ell,i)=\binom{\ell+1}{i}C_{\ell}-a(\ell-1,i-2)
\] 
for $i \ge 2$, $a(\ell,0)=C_{\ell}$, and $a(\ell,1)=(\ell+1)C_{\ell}$. Translating the last two conditions, we need to show that $b_{\ell}(\Borel(x_1 \cdots x_{\ell+1}))=C_{\ell}$ and $b_{\ell-1}(\Borel(x_1 \cdots x_{\ell+1}))=(\ell+1)C_{\ell}$. Both of these formulas follow from the Eliahou-Kervaire resolution and the counts in the previous section. Rewriting the recursive formula for the $a(\ell,i)$ in the new indexing, we also need to prove that
\[
b_{\ell-i}(\Borel(x_1 \cdots x_{\ell+1})) = \binom{\ell+1}{i}C_{\ell} - b_{\ell-i+1}(\Borel(x_1 \cdots x_{\ell})).
\]
This follows from Corollary~\ref{c:firstvars} (rewriting $\binom{\ell+1}{i}$ as $\binom{\ell+1}{\ell-i+1}$).
\end{proof}

In light of the surprising connection in Theorem~\ref{e:ppt}, we close with two questions:

\begin{question}\label{q:combinatorialbi}
Is there a nice combinatorial bijection between some basis of the minimal resolution of $\Borel(x_1 \cdots x_n)$ and the pointed pseudo-triangulations of the single chain?  (There exist other bases for the resolution with interesting topological structure; see, for example, \cite{NR,Si}.)
\end{question}

\begin{question}\label{q:singlechain}
Suppose we change the set $A$ to a different point configuration. Do the analogous $a(\ell,i)$ correspond to the Betti numbers of other (in special cases, possibly principal) Borel ideals?
\end{question}


\end{document}